\documentclass{article}

\usepackage{amssymb}
\usepackage{mathtools}
\usepackage{amsmath}
\usepackage{hyperref}
\usepackage[english,francais]{babel}
\usepackage{graphicx}
\newtheorem{theorem}{Theorem}[section]
\newenvironment{proof}{\it Proof.}{\hfill$\square$}
\newtheorem{lemma}[theorem]{Lemma}
\newtheorem{e-proposition}[theorem]{Proposition}

\newtheorem{e-definition}[theorem]{Definition\rm}

\newtheorem{theoreme}{Th\'eor\`eme}[section]

\newtheorem{corollaire}[theoreme]{Corollaire}

\def\og{\leavevmode\raise.3ex\hbox{$\scriptscriptstyle\langle\!\langle$~}}
\def\fg{\leavevmode\raise.3ex\hbox{~$\!\scriptscriptstyle\,\rangle\!\rangle$}}

\title{Inequalities for the generalized numerical radius}
\author{H. Abbas \\ \texttt{habbas@ul.edu.lb} \and S. Harb  \\ \texttt{sadeem.harb3@gmail.com} \and H. Issa  \\ \texttt{hassan.issa@mathematik.uni-goettingen.de}}

\begin{document}




\selectlanguage{english}

\maketitle


\selectlanguage{english}




\begin{abstract}
\selectlanguage{english}
In the present paper, we  provide several inequalities for the  generalized numerical radius of  operator matrices as introduced by A. Abu-omar and F. Kittaneh  in \cite{kittaneh}. We generalize  the convexity and the log-convexity results obtained by M. Sababheh in \cite{sababheh} for the case of the numerical radius   to the case of  the generalized numerical radius.  We illustrate our work by providing a positive answer for the question addressed  in \cite{sababheh} for the convexity of a certain matrix operator function.  Moreover, and motivated by the results of A. Aldalabih and F. Kittaneh in \cite{kittaneh0} for the case of Hilbert-Schmidt numerical radius norm,  we use  some Schatten $p$-norm inequalities for partitioned $2\times 2$ block-matrices  to provide several Schatten $p$-norm numerical radius inequalities. 

\end{abstract}

\selectlanguage{francais}

\selectlanguage{english}
\section{Introduction}
\label{S1}
It is well known that the   numerical radius  plays an important role in various fields of operator theory  and matrix analysis (cf. \cite{ando,goldberg,halmos}).  Based on some operator theory studies on Hilbert spaces, several generalizations for the concept of numerical radius have recently been introduced \cite{bo,kittaneh0,zamani}. In a paper of Abou-omarr and F. Kittaneh \cite{kittaneh0} the authors introduced the so-called generalized numerical radius. If $H$ is a separable Hilbert space  and  $N$ is any norm on the space of bounded operators $\mathcal{B}(H)$, the generalized numerical radius for $A\in\mathcal{B}(H)$, denoted by $w_{N}( A )$, is obtained via  the supremum of the norm over the real parts of all rotations of $A$ i.e.   $$w_{N}( A ) =  \displaystyle \sup_{\theta \in \mathbb{R} }  N( Re( e^{i \theta} A )).$$ 
Simple computation shows that when $N$ is the usual operator norm inherited from the inner product on $H$ then $w_N$ coincides with the usual numerical radius norm $w(\cdot)$. We refer the reader to \cite{kittaneh} and \cite{sababheh} for intermediate properties and  inequalities of the  norm $w_{N}( . )$.

\bigskip
In the present paper, we  restrict our attention to operator matrices i.e. $A\in M_n(\mathbb C)$ and we provide several inequalities for the matrix norm $w_{N}( . )$. Some results are obtained via convexity  whenever $N$ is unitarily invariant norm. Other results are restricted to the case when $N$ is the Schatten $p$-norm and are obtained via some norm inequalities of $2\times 2$ partitioned block-matrices.

\bigskip
On the one hand, we follow up the work of M. Sababheh in \cite{sababheh} for the case of the numerical radius, to establish a new Young-type inequality for $w_{N}( . )$.  Addressing to an open question proposed by the author in \cite{sababheh}  about the convexity of the function $t \mapsto w(A^{t}XA^{t})$, here $A>0$ denotes a positive definite matrix, on $\mathbb{R}$ we prove that the convexity is not only true for $w$ but remains true for $w_{N}$. On the other hand, we use the results of R. Bhatia and F. Kittaneh  in \cite{bh} to establish  upper bounds for the Schatten $p$-generalized numerical radius of a partitioned $2\times2$ block matirx  with respect to the Schatten $p$-generalized numerical radius of the diagonal part and the Schatten $p$-norm of the off-diagonal parts.

\bigskip
Below we state the first result of this note which generalizes Theorem 2.2 and Theorem 2.3 in \cite{sababheh}. 
\begin{theorem}
Let $ N(.) $ be a unitarily invariant norm on $M_n(\mathbb{C})$.	Given a positive definite matrix $ A > 0$ and an arbitrary matrix  $X \in \mathbb{M}_{n}(\mathbb{C})$. Each of the following  functions
\begin{itemize}
	\item $f(t) = w_{N}( A^{t}XA^{1 - t} + A^{1 - t}XA^{t})$ and
	\item $g(t) = w_{N}( A^{t}XA^{1 - t}).$
\end{itemize}
are  convex  over $ \mathbb{R}.$ 
\end{theorem}

\noindent

In connection to the work in \cite{kittaneh0} for the Hilbert-Schmidt generalized numerical radius norm, we use the following notation: given $p\in [1,\infty[$ we consider $N=\lVert\cdot \rVert _p$, the Schatten $p$-norm,  and we denote by $w_p$ the corresponding generalized numerical radius. We obtained the following estimation:

\begin{theorem}
	Let $ T =  \bigg[ \begin {array}{cc}
	A & B\\
	C & D\\
	\end {array} \bigg]\in M_{2n}(\mathbb C)$ then for any  $p\in [2,\infty[$ we have
\begin{equation}\label{e1}
w^p_{p}( T )\ \ \leq \ \ \frac{1}{2^{2-p}} \Big( w_{p}^{p}( A ) + w_{p}^{p}( D ) + \frac{1}{2^{p - 1}} \big[ \parallel B \parallel_{p} + \parallel C \parallel_{p} \big ]^{p} \Big). 
\end{equation}  
	Moreover, if $p\in[1,2]$ then 
	\begin{equation}\label{e2}
	 w^p_{p}( T )\ \ \leq \Big( w_{p}^{p}( A ) + w_{p}^{p}( D ) + \frac{1}{2^{p - 1}} \big[ \parallel B \parallel_{p} + \parallel C \parallel_{p} \big ]^{p} \Big).
	\end{equation}  
\end{theorem}

\bigskip

We point out that a lower bound for the Schatten $p$-generalized numerical radius of an operator with respect to the Schatten $p$-norm of the operator has already been established by T. Bottazi and C. Conde in \cite{co}. Indeed, using a Clarkson inequality obtained by O. Hirzalla and F. Kittaneh in \cite{hirz} it follows directly that inequality (\ref{e1}) is bounded below by $\frac{1}{2^{p - 1}} \parallel T \parallel^p_{p}$ and (\ref{e2}) is bounded below by $\frac{1}{2} \parallel T \parallel^p_{p}$. 

\section{Sketch of the proof of Theorem 1.1}
Throughout this section, $N$ denotes a unitarily invariant norm on $M_n(\mathbb{C})$. We aim to establish the convexity in Theorem 1.1. using some type  of H\"{o}lder and Heinz inequalities. It should be mentioned that $w_N$ is  in general not unitarily invariant but rather weakly unitarily invariant. For such reason, we need to investigate such types of inequalities in  a weaker condition (cf. Ch. IV in \cite{Bhatia} for the case of unitarily invariant norms). 

 \begin{lemma}\label{3l1}
 	Let $ A, X  \in  \mathbb{M}_{n}(\mathbb{C})  $ such that $ A >0 $. Given $t\in[0,1]$, the following estimates hold
 	\begin{align}
 	& w_{N}( A^tXA^t ) \leq \ \ w^{t}_{N}( AXA ) w^{1-t}_{N}(X),\label{1}\\ 
 	& 2 w_{N}( A^{\frac{1}{2}}XA^{\frac{1}{2}} ) \leq  w_{N}( A^tXA^{1 - t} +  A^{1 - t}XA^{t} ) \leq  w_{N}( AX + XA ).\label{2}
 	\end{align}
 \end{lemma}
\begin{proof}
As $A^t$ is Hermitian then for any $\theta \in \mathbb{R},$	it follows that $Re( e^{i \theta} A^tXA^t ) = A^t Re( e^{i \theta}X ) A^t. $ Now applying H\"{o}lder inequality for the unitarily invariant norm N (cf. for example Corollary IV.5.4 in \cite{Bhatia}), we get \begin{equation*}
N( Re( e^{i \theta} A^{t}XA^{t} )  )   \leq  N^{t}( A Re( e^{i \theta} X ) A ) N^{1 - t}(  Re( e^{i \theta} X )).
\end{equation*}
Taking the supremum over all  $\theta$ in $\mathbb{R},$ the inequality in (\ref{1}) follows. A similar proof holds for (\ref{2}) by using Corollary IV.4.9 in \cite{Bhatia}.
\end{proof}

We point out that replacing $ A $ by $ A^{2} $ and for $ t = 1 $ in (\ref{2}) we get
\begin{equation} w_{N}( AXA ) \leq \frac{1}{2} w_{N}( A^{2}X + XA^{2} ). \end{equation}
Hence, for $ t , s  \in  \mathbb{R}$ we obtain   
 \begin{align*}  f( \frac{t + s}{2} )& =  w_{N}( A^{\frac{t - s}{2}} ( A^{s}XA^{1 - t} + A^{1 - t}XA^{s} ) A^{\frac{t - s}{2}} ) \\
& \leq  \frac{1}{2}  w_{N}( A^{t - s} ( A^{s}XA^{1 - t} + A^{1 - t}XA^{s} ) + ( A^{s}XA^{1 - t} + A^{1 - t}XA^{s} ) A^{t - s} )\ \  \\ \vspace{0.1cm}
 & \leq \frac{1}{2}  w_{N}( A^{t}XA^{1 - t} + A^{1 - t}XA^{t} ) + \frac{1}{2}  w_{N}( A^{s}XA^{1 - s} + A^{1 - s}XA^{s} ) =  \frac{1}{2}  f(t) + \frac{1}{2}  f(s).
\end{align*}

We note that  the first inequality of (\ref{2}) ensures that $f$  admits a global minimum at  $ t = \frac{1}{2} $. The proof for the convexity of  $g$ follows by a similar argument. 
\section{Sketch of proof of Theorem 1.2}

Following the notation in \cite{bh},  we recall the following Schatten $p$-norm estimates of an operator to that of its $2\times2$ block matrix entries.
\begin{lemma}
Let $T=  \big[ T_{ij} \big]$, $1\leq i,j\leq 2$ be a block matrix. Then for any $p\in[2,\infty[$ the following holds

\begin{equation}\label{3}
\parallel T \parallel^p_{p}  \leq  \frac{1}{2^{\frac{2}{p}} - 1} \big( \displaystyle \sum_{i,j} \parallel T_{i,j} \parallel_{p}^{p} \big).
\end{equation}
\end{lemma}

Let $ T =  \bigg[ \begin {array}{cc}
A & B\\
C & D\\
\end {array} \bigg]\in M_{2n}(\mathbb{C})$, then for  any $\theta \in \mathbb{R} $ we have $ Re( e^{i \theta} T) =  \bigg[ \begin {array}{cc}
Re(e^{i \theta} A )                                   & F \\
F^\star & Re( e^{i \theta} D )                                \\
\end {array} \bigg], $
where $F=\frac{1}{2}( e^{i \theta} B + e^{- i \theta } C^* ).$ Applying (\ref{3}) we obtain  
\begin{align*}
\parallel Re( e^{i \theta} T ) \parallel^p_{p}& \leq  \frac{1}{2^{2-p}} \Big( \parallel Re( e^{i \theta } A ) \parallel_{p}^{p} + \parallel Re( e^{i \theta} D ) \parallel_{p}^{p} + 2 \parallel F \parallel_{p}^{p})\\
&\leq \ \ \frac{1}{2^{2-p}} \Big( \parallel Re( e^{i \theta } A ) \parallel_{p}^{p} + \parallel Re( e^{i \theta} D ) \parallel_{p}^{p} + \frac{1}{2^{p - 1}} \big( \parallel B \parallel_{p} + \parallel C \parallel_{p} \big)^{p} \Big)\\
&\leq  \frac{1}{2^{2-p}} \Big( w_{p}^{p}( A ) + w_{p}^{p}( D ) + \frac{1}{2^{p - 1}} \big[ \parallel B \parallel_{p} + \parallel C \parallel_{p} \big ]^{p} \Big).
\end{align*}
Taking the supremum over all $\theta \in \mathbb{R} $ inequality (\ref{e1}) is obtained.  Using the corresponding results in \cite{kittaneh0},  for the case $p\in[1,2]$, inequality (\ref{e2}) is established in a similar way.

 
 \section{Final remarks}
 
 In this section, we provide further results  with some applications of Theorems 1.1 and 1.2.
 
 \bigskip
Motivated by the work of Shabebh for the numerical radius, the convexity of the function $f$, together with the convexity of $\ell(t):=  w_{N}( A^{t}XA^{1 - t} - A^{1 - t}XA^{t} )$,  as defined in Theorem 1.1 provides the following   reversed inequalities for the generalized numerical radius.
\begin{corollaire}
Let $N$ be a unitarily inavariant matrix norm and let $ A > 0$. For any  $  X  \in \mathbb{M}_{n}(\mathbb C)$ the following inequalities hold
 	\begin{equation*}\label{5012} 
		\begin{cases}
		w_{N}( A^{t}XA^{1 - t} \pm A^{1 - t}XA^{t} )  \leq  w_{N}( AX \pm XA ),  &\mbox{ if } t \in [0 , 1] \\
		w_{N}( A^{t}XA^{1 - t} \pm A^{1 - t}XA^{t} ) \geq w_{N}( AX \pm XA ), &\mbox{ if  } t \notin [0 , 1] .
		\end{cases}
		\end{equation*}
		\
\end{corollaire}
The convexity of the function $ g$ defined in Theorem 1.1 provides the following  Young-type inequality which generalizes Corollary 2.11 in  \cite{Sababheh}.
\begin{corollaire}
Let $N$ be a unitarily inavariant matrix norm. Given $A>0 $ and  $X  \in \mathbb{M}_{n}(\mathbb C)$ then
	\begin{equation*}\label{504} 
	\begin{cases}
	w_{N}( A^{t}XA^{1 - t} )  \leq  t  .  w_{N}(AX) + (1 - t) . w_{N}( XA )   & \mbox{ if } t \in [0 , 1],\\
	w_{N}( A^{t}XA^{1 - t} ) \geq  t . w_{N}(AX) + (1 - t) . w_{N}(XA)   &  \mbox{ if } t  \not\in  [0 , 1].
	\end{cases}
	\end{equation*}
\end{corollaire}
 Given $A > 0,$  and  $X  \in  \mathbb{M}_{n}(\mathbb{C}),$ the author in \cite{Sababheh} proposed a question on the convexity of $t \mapsto w(A^{t}XA^{t})$ for the numerical radius and indicated that he has no answer whether this is true or not. Below, we provide a positive answer for the aforementioned questioned not only for the numerical radius but also for the generalized numerical radius whenever $N$ is unitarily invariant.
 
 \bigskip
 Indeed,  by  Theorem IX.4. 8 in \cite{Bhatia} the inequality $2N(AXB)\leq N(A^2X+XB^2)$ holds true for all positive matrices $A$ and $B$. Following a smilar argument to that in Lemma 2.1, we get $2w_N(AXB)\leq w_N(A^2X+XB^2)$. In particular, if $ h(t) = w_{N}( A^tXA^t )$ then
 $$ h(\frac{t + s}{2}) = w_N(A^{\frac{t - s}{2}} ( A^{s}XA^{t} ) A^{\frac{- t + s}{2}} )\leq \frac{1}{2}w_{N}( A^tXA^t + A^sXA^s) = \frac{1}{2} h(t) + \frac{1}{2} h(s).$$
We point out that by choosing $t=\frac{1}{2}$ in (\ref{1}), the log-convexity of $h$ is also obtained. Moreover, and by repeating a similar argument we get
 	\begin{corollaire}
 		Let $N$ be a unitarily inavariant matrix norm. Given $A>0 $ and  $X  \in \mathbb{M}_{n}(\mathbb C)$ the function \begin{equation*}
 		k(t) =  w_{N}( A^{t}XA^{1 - t} + A^{-t}XA^{1 + t} )
 		\end{equation*}
 		is convex on $\mathbb{R},$ with  minimum  attained at $ t  =  0$.
 	
 	\end{corollaire}

\bigskip
Adressing to the consequences of Theorem 1.2 and following a similar logic to that in \cite{kittaneh0} we obtain the following results
 	\begin{corollaire} Let $p\in[2,\infty[$ then the following estimates holds true
 		\begin{enumerate}
 			\item $w_p\bigg(\bigg[ \begin {array}{cc}
 			A & 0\\
 			0 & D\\
 			\end {array} \bigg]\bigg)\leq \frac{1}{2^{\frac{2}{p} - 1}} \Big( w_{p}^{p}( A ) + w_{p}^{p}( D ) \Big)^{\frac{1}{p}},$ and $w_p\bigg(\bigg[ \begin {array}{cc}
 			A & 0\\
 			0 & D\\
 			\end {array} \bigg]\bigg)= \big( w_{p}^{p}( A ) + w_{p}^{p}( D ) \big)^{\frac{1}{p}}$ if $A$ and $D$ are Hermitian.
 			\item $w_p\bigg(\bigg[ \begin {array}{cc}
 			A & B\\
 			0 & 0\\
 			\end {array} \bigg]\bigg)\leq  \frac{1}{2^{\frac{2}{p} - 1}} \Big( w_{p}^{p}( A ) + \frac{1}{2^{p - 1}} \parallel B \parallel_{p}^{p} \Big)^{\frac{1}{p}}.$
 		\end{enumerate}
\end{corollaire}
Using the unitarily invariance of Schatten $p$-norms together with the proof of Theorem 1.1 we get a generalization for Lemma 2 in \cite{kittaneh0}:
\begin{corollaire} Let $p\in[2,\infty[$ then the following are true
	\begin{enumerate}
		\item $w_p\bigg(\bigg[ \begin {array}{cc}
		0 & B\\
		B & 0\\
		\end {array} \bigg]\bigg)=2^\frac{1}{p}  w_{p}^{p}( B )$.
		\item $w_p\bigg(\bigg[ \begin {array}{cc}
		A & B\\
		B & A\\
		\end {array} \bigg]\bigg)\leq \frac{1}{2^{\frac{2}{p} - 1}} \Big( w_{p}^{p}( A+B ) + w_{p}^{p}( A-B ) \Big)^{\frac{1}{p}},$ and $w_p\bigg(\bigg[ \begin {array}{cc}
		A & B\\
		B & A\\
		\end {array} \bigg]\bigg)= \Big( w_{p}^{p}( A+B ) + w_{p}^{p}( A-B ) \Big)^{\frac{1}{p}}$ if $A$ and $B$ are Hermitian.
	\end{enumerate}
\end{corollaire}
Finally, we note that similar inequalities to those obtained in the previous two corollaries can be derived for the case $p\in[1,2]$ using (\ref{e2}). Moreover, such results  remain true for the generalized numerical radius on bounded operators acting on separable Hilbert spaces.





\end{document}